\newcommand{\XX}{\mathbb{X}}

\newcommand{\Knull}[1]{\mathrm{K}_0(#1)}
\newcommand{\Knullprime}[1]{\mathrm{K}'_0(#1)}
\newcommand{\rKnull}[1]{\overline{\mathrm{K}}_0(#1)}
\newcommand{\coh}[1]{\mathrm{coh}\,#1}
\newcommand{\coker}[1]{\mathrm{cok}(#1)}
\newcommand{\cohnull}[1]{\mathrm{coh}_0#1}
\newcommand{\vect}[1]{\mathrm{vect}{#1}}
\newcommand{\up}[1]{\stackrel{#1}{\longrightarrow}}
\newcommand{\Pic}[1]{\mathrm{Pic}(#1)}
\newcommand{\Picnull}[1]{\mathrm{Pic}_0(#1)}
\newcommand{\Div}[1]{\mathrm{Div}(#1)}

\newcommand{\Hom}[2]{\mathrm{Hom}(#1,#2)}
\newcommand{\Ext}[3]{\mathrm{Ext}^{#1}({#2},{#3})}
\newcommand{\euform}[2]{\langle #1,#2\rangle}
\newcommand{\aveuform}[2]{\llangle #1,#2\rrangle}
\newcommand{\wt}[1]{\langle #1 \rangle}

\newcommand{\rk}[1]{\mathrm{rk}\,#1}
\newcommand{\dg}[1]{\mathrm{deg}\,#1}
\newcommand{\Aut}[1]{\mathrm{Aut}(#1)}

\newcommand{\ZZ}{\mathbb{Z}}
\newcommand{\CC}{\mathbb{C}}
\newcommand{\HH}{\mathbb{H}}

\newcommand{\PP}{\mathbf{P}}
\newcommand{\Cl}[1]{\mathcal{C}\ell(#1)}
\newcommand{\Cc}{\mathcal{C}}
\newcommand{\Hh}{\mathcal{H}}
\newcommand{\Hhnull}{\mathcal{H}_0}
\newcommand{\tHh}{\widetilde{\mathcal{H}}}
\newcommand{\Ll}{\mathcal{L}}
\newcommand{\LL}{\mathbb{L}}

\newcommand{\si}{\sigma}
\newcommand{\mmod}[1]{\mathrm{mod}(#1)}
\newcommand{\modgr}[2]{\mathrm{mod}^{#1}(#2)}

\newcommand{\dual}[1]{D(#1)}
\newcommand{\al}{\alpha}

\newcommand{\wtilde}{\widetilde}
\newcommand{\OX}{\mathcal{O}_X}
\newcommand{\OXX}{\mathcal{O}_\XX}
\newcommand{\rcl}[1]{\lsem #1\rsem}
\newcommand{\End}[2]{\mathrm{End}_{#1}(#2)}
\newcommand{\lra}{\longrightarrow}
\newcommand{\Ee}{\mathcal{E}}
\newcommand{\funct}[1]{k(#1)}
\newcommand{\divi}{\mathrm{div}}
\newcommand{\nat}{\mathrm{nat}}
\newcommand{\rperp}[1]{#1^\perp}
\newcommand{\tf}{\tilde{f}}
\newcommand{\vx}{\vec{x}}

\documentclass{amsart}

\usepackage{amssymb}

\usepackage[cmtip,all]{xy}

\usepackage{MnSymbol} 

\newtheorem{theorem}{Theorem}[section]
\newtheorem{lemma}[theorem]{Lemma}
\newtheorem{proposition}[theorem]{Proposition} 
\newtheorem{corollary}[theorem]{Corollary} 
\theoremstyle{definition}
\newtheorem{definition}[theorem]{Definition}
\newtheorem{example}[theorem]{Example}

\theoremstyle{remark}
\newtheorem{remark}[theorem]{Remark}

\numberwithin{equation}{section}

\begin{document}

\title[K-theory of weighted curves]{On the K-theory of weighted projective curves}


\author[H. Lenzing]{Helmut Lenzing}
\address{University of Paderborn\\
Institute for mathematics\\
D-33098 Paderborn, Germany}
\email{helmut@math.uni-paderborn.de}
\thanks{}


\date{}

\begin{abstract}
We present a largely self contained account on the K-theory of a weighted smooth projective curve over an algebraically closed field. In particular, we discuss the weighted version of divisor theory, Euler form, and Riemann-Roch theorem. This includes a treatment of the orbifold Euler characteristic.
\end{abstract}

\maketitle

\setcounter{tocdepth}{2}
\tableofcontents
\section*{Introduction}
Throughout we work over an algebraically closed base field $k$. This article deals with the K-theory of weighted smooth projective curves (weighted projective curves, for short). By definition, a weighted projective curve $\XX$ is a pair $(X,w)$, consisting of a smooth projective curve $X$ and a weight function $w$ which is an integral-valued positive function on $X$ such that $w(x)>1$ holds for only finitely many (closed) points $x_1,\ldots,x_t$ of $X$. We call the integers $p_i=w(x_i)$, $i=1,\ldots,t$, the \emph{weights} of $\XX$, and write $\XX=X\begin{pmatrix}
  p_1,p_2,\ldots,p_t \\
  x_1,x_2,\ldots,x_t
\end{pmatrix}$ or simply $\XX=X\wt{p_1,p_2,\ldots,p_t}$ in contexts where the position of the weights does not matter.

Weighted projective curves, in particular weighted projective lines \cite{Geigle:Lenzing:1986}, where $X$ is the projective line $\PP^1$, occur in many different incarnations: as compact Riemann surfaces, equipped with a weight function~\cite{Lenzing:2016b}, as finite group quotients of smooth projective curves (good characteristic assumed), as smooth projective curves equipped with a parabolic structure~\cite{Seshadri:1982}, as holomorphic orbifolds, as one dimensional Deligne-Mumford curves (stacks)~\cite{Behrend:Noohi:2006}, \cite{Abdelgadir:Ueda:2015} and, for the base field of complex numbers, as finite (orbifold) group quotient of a compact Riemann surface, see \cite{Lenzing:2016b}.

Instead of dealing with any of the particular incarnations, our focus lies on the category of coherent sheaves associated with a weighted projective curve, which forms a bridge between the different incarnations. By work of Reiten and van den Bergh~\cite{Reiten:VandenBergh:2002}, see also \cite{Lenzing:Reiten:2006}, \cite{VanRoosmalen:2016}, \cite{Lenzing:2007} the category of coherent sheaves on a weighted projective curve $\XX$ can be characterized to be a Hom-finite (over $k$) hereditary noetherian $k$-category $\Hh$ with Serre duality in the form $\Ext{1}{A}{B}=\dual{\Hom{B}{\tau A}}$, where $\tau$ is a self-equivalence of $\Hh$ and where the function field $K=k(\Hh)$ of $\Hh$ has infinite $k$-dimension. The category of coherent sheaves on $(X,w)$ can, for instance, be constructed by means of a sheaf of hereditary orders over the underlying smooth projective curve $X$~\cite{Reiten:VandenBergh:2002}, see also \cite{Iyama:Lerner:2014}. For the present article we apply instead  the $p$-cycle construction of \cite{Lenzing:1998} to the category of coherent sheaves $\coh{X}$ which makes the K-theoretic implications immediately transparent.

We finally note that for the case of weighted projective lines there exists already detailed accounts on their K-theory, see~\cite{Lenzing:1996}, \cite{Kussin:2000} which however do not cover divisor theory.

\section{Coherent sheaves on a smooth projective curve}

Let $X$ be a smooth projective curve over $k$. By $\Hh=\coh{X}$ we denote its category of coherent sheaves and by $\Hhnull=\cohnull{X}$ its full subcategory of sheaves of finite length. Let $\OX$ be a line bundle of $X$ that we call structure sheaf of $X$. The $k$-dimension of $\Ext{1}{\OX}{\OX}$ is called the \emph{genus} of $X$. An even more important invariant of $X$ is the function field $K=k(X)$ of $X$ given by the equivalence $\Hh/\Hhnull=\mmod{K}$,
where $\mmod{K}$ denote the category of finite dimensional $K$-vector spaces and where $\tHh=\Hh/\Hhnull$ is the Serre quotient~\cite{Gabriel:1967} of $\Hh$ modulo its Serre subcategory $\Hhnull$.
Each object in this quotient category has the form $\OX^n$ for some integer $n$; moreover, the endomorphism ring of $\OX$ in $\tHh$ is a commutative field $K$, which is a function field $K$ of one variable over $k$.

We next collect some fundamental facts about the category $\coh{X}$. By definition, a coherent sheaf $E$ is called a \emph{vector bundle} if $E$ does not have any simple subobject. By $\vect{X}$ we denote the full subcategory of $\coh{X}$ formed by vector bundles. A \emph{line bundle} is a vector bundle of rank one, and necessarily indecomposable.

\begin{theorem}\label{thm:CohX}
The category $\Hh=\coh{X}$ has the following properties.
\begin{enumerate}
  \item $\Hh$ is a noetherian abelian Hom-finite category having Serre duality in the form $\Ext{1}{A}{B}=\dual{\Hom{B}{\tau A}}$ for a self-equivalence $\tau$ of $\Hh$.
  \item $\Hh$ is also hereditary and has almost-split sequences with $\tau$ serving as the Auslander-Reiten translation.
  \item Each coherent sheaf $A$ is the direct sum $E\oplus U$ of a vector bundle $E$ and a coherent sheaf $U$ of finite length.
  \item The Auslander-Reiten translation $\tau$, when restricted to the full subcategory $\cohnull{X}$ of coherent sheaves of finite length, is isomorphic to the identity functor.
  \item Each indecomposable sheaf $U$ of finite length $\ell$ is uniserial, accordingly uniquely determined by its simple socle $S$ and its length. Each simple sheaf $S$ is concentrated in a uniquely determined point $x\in X$. Notation: $S_x$.
  \item Each vector bundle of rank $r$ has a finite filtration $E=E_r\supset E_{r-1}\supset \ldots \supset E_1\supset E_0=0$ with line bundle factors $E_i/E_{i-1}$.\qed
\end{enumerate}
\end{theorem}
For the proof we refer to \cite{Hartshorne:1977} or, taking the view of hereditary categories, to \cite{Lenzing:2007}. Also the arguments of \cite{Geigle:Lenzing:1987} apply. Here, we just remark that the rank of a coherent sheaf $A$ may be defined as the length of $A$ in the Serre quotient $\coh{X}/\cohnull{X}$.

\subsection{The Euler form}
As for any Hom- and Ext-finite hereditary abelian category, the \emph{Euler form} for $\coh{X}$ is the bilinear form on the Grothendieck group $\Knull{X}=\Knull{\coh{X}}$, given on classes of objects by the expression
\begin{equation}\label{eqn:EulerForm}
  \euform{[A]}{[B]}=\dim_k\Hom{A}{B}-\dim_k\Ext1{A}{B},
\end{equation}
where often we simply write $\euform{A}{B}$.
Due to Serre duality we further have
\begin{equation}\label{eqn:antisymmetry}
  \euform{b}{a}=-\euform{a}{\tau b} \text{ for all members }a, b \text{ from }\Knull{X}.
\end{equation}
In particular, we have $\euform{a}{-}=0$ if and only if $\euform{-}{a}=0$.
\begin{lemma}
The kernel $N=\{a\in\Knull{X}\mid \euform{a}{-}=0\}$ of the Euler form (or of $\Knull{X}$) is generated by the differences $[S_x]-[S_y]$ ($x,y\in X$) of the classes of all pairs of simple sheaves.
\end{lemma}
\begin{proof}
Using the following properties
\begin{enumerate}
  \item The class of a vector bundle is a sum of classes of line bundles.
  \item Simple sheaves, concentrated in different points, are Hom-Ext-orthogonal.
  \item For each line bundle $L$ the difference $[L]-[\OX]$ belongs to the subgroup of $\Knull{X}$ generated by the classes of simple sheaves.
\end{enumerate}
the proof is straightforward.
\end{proof}
Clearly, the Euler form descends to the \emph{reduced} (or \emph{numerical}) Grothendieck group $\rKnull{X}=\Knull{X}/N$ and there yields a non-degenerate bilinear form. The class of a coherent sheaf $A$ in $\rKnull{X}$ we denote by double brackets $\rcl{A}$.
\begin{proposition}\label{prop:redGrothendieckGroup}
  The reduced Grothendieck group is the free abelian group on the class $\rcl{\OX}$ of the structure sheaf and the class $\rcl{S_x}$ of any simple sheaf. Moreover, in this basis the Euler form is given by the following \emph{Cartan matrix}
\begin{equation}\label{eqn:CartanMatrix}
  \begin{pmatrix}
  \euform{\OX}{\OX}&\euform{\OX}{S_x}\\
  \euform{S_x}{\OX}&\euform{S_x}{S_x}
  \end{pmatrix}=
  \begin{pmatrix}
  1-g& 1 \\
  -1 & 0
  \end{pmatrix},
\end{equation}
where $g$ is the genus of $X$.
 \end{proposition}

If $C$ denotes the Cartan matrix, the  matrix $\Phi=-C^{-1}C^{tr}$ is called \emph{Coxeter matrix}.
\begin{corollary}
The Auslander Reiten translation $\tau$ induces on the reduced Grothendieck group the Coxeter transformation, also denoted $\tau$, given by the Coxeter matrix
  \begin{equation}\label{eqn:CoxeterMatrix}
    \begin{pmatrix}
      1 & 0 \\
      2(g-1) & 1
    \end{pmatrix}
  \end{equation}
whose characteristic polynomial, the \emph{Coxeter polynomial} of $X$, equals $(x-1)^2$.
  In particular, $\rcl{\tau\OX}=\rcl{\OX}+2(g-1)\rcl{S_x}$ where $x$ is any point of $X$.
\end{corollary}
We note in this context that $2(1-g)$ is the \emph{Euler characteristic} of $X$.

By Proposition~\ref{prop:redGrothendieckGroup} we obtain two ``obvious'' linear forms on the Grothendieck group, \emph{rank} and \emph{degree} given by the expressions
\begin{equation}\label{eqn:RankAndDegree}
\rk=\euform{-}{S_x} \textrm{ and } \dg=\euform{\OX}{-}-\euform{\OX}{\OX}\rk,
\end{equation}
respectively. We note first that the above K-theoretic definition of rank agrees for coherent sheaves with the sheaf theoretic one, given by the length of $X$ in the Serre quotient $\coh{X}/\cohnull{X}$: Clearly, the sheaf theoretic definition induces a linear form on $\rKnull{X}$ which agrees with the K-theoretic one on the basis $\rcl{\OX}$, $\rcl{S_x}$ of $\rKnull{X}$.

Rank and degree have largely complementary properties:
\begin{proposition}
\begin{enumerate}
  \item For each coherent sheaf $A$ we have $\rk{A}\geq0$; moreover, $\rk{A}=0$ if and only if $A$ has finite length.
  \item We have $\dg{\OX}=0$. Moreover, for each coherent sheaf $A$ of finite length $\dg{A}$ equals the length of $A$. Hence for members of $\cohnull{X}$ we have $\dg{A}\geq0$ and $\dg{A}=0$ if and only if $A=0$.
\end{enumerate}
\end{proposition}
\begin{proof}
  The first assertion follows from the fact that the rank of a coherent sheaf $X$ equals its length in the Serre quotient $\Hh/\Hhnull$. For the second assertion we use that $\Hom{\OX}{S_x}=k$ and, due to Serre duality, $\Ext{1}{\OX}{S_x}=0$ implying  $\euform{\OX}{S_x}=1$ for each simple sheaf $S_x$.
\end{proof}
\begin{proposition}[Riemann-Roch]
  For any two coherent sheaves $A, B$ on $X$ we have
  \begin{equation}\label{eqn:RiemannRoch}
    \euform{A}{B}=(1-g)\cdot\rk{A}\cdot\rk{B}+\begin{vmatrix}
                                      \rk{A} & \rk{B} \\
                                      \dg{A} & \dg{B}
                                    \end{vmatrix}.
  \end{equation}
\end{proposition}
\begin{proof}
  It suffices to check the formula on the basis $\rcl{\OX}$, $\rcl{S_x}$ of the reduced Grothendieck group $\rKnull{X}$, where its validity is obvious.
\end{proof}
We recall that in a hereditary abelian $k$-category $\Hh$ an object $E$ is called \emph{exceptional} if $\End{}{E}=k$ and $\Ext{1}{E}{E}=0$.
\begin{corollary}
The category $\coh{X}$ on a smooth projective curve $X$ has an exceptional object only for genus zero, in which case exactly the line bundles are exceptional.
\end{corollary}

\begin{proof}
  Assume $E$ is exceptional in $\coh{X}$, then by the  Riemann-Roch formula we obtain $1=\euform{E}{E}=(1-g)(\rk{E})^2$. This implies $g=1$ and $\rk{E}=1$. Conversely, for genus zero, the structure sheaf, hence each line bundle, is exceptional.
\end{proof}

\begin{proposition} \label{prop:LineBundleClass}
A line bundle $L$ is determined up to isomorphism by its class $[L]$ in $\Knull{X}$.
\end{proposition}
\begin{proof}
The proof is classical, and contained in the paper \cite{Borel:Serre:1958} that marks the beginning of algebraic K-theory: Let $\Pic{X}$ be the group of isomorphism classes of line bundles with respect to the tensor product of coherent sheaves on $X$. Following~\cite{Borel:Serre:1958}, mapping the class $[E]$ of a vector bundle $E$ of rank $r$ to the line bundle $\bigwedge^rE$, induces the \emph{determinant homomorphism}  $\det:\Knull{\coh{X}}\to \Pic{X}$, sending the class $[L]$ of a line bundle $L$ to the isomorphism class of $L$, thus  implying the claim.
\end{proof}

By Serre duality we have $\Ext{1}{S_x}{\OX}=k$ for each $x$ in $X$. There results a unique non-split exact sequence $0\to \OX \up{\al_x} \OX(x)\to S_x\to 0$, where $\OX(x)$ is a line bundle.
\begin{theorem}
  For a smooth projective curve we have the following alternative:
  \begin{enumerate}
    \item For genus zero, all simple sheaves have the same class in $\Knull{X}$. Further, the space $\Hom{\OX}{\OX(x)}$ has dimension two for each $x\in X$.
    \item For genus greater zero, non-isomorphic simple sheaves have different classes in $\Knull{X}$. Further, for each $x$ in $X$, the space $\Hom{\OX}{\OX(x)}$ has dimension one.
  \end{enumerate}
\end{theorem}
\begin{proof}
It suffices to deal with the second assertion. Assume that the simple sheaves $S_x$ and $S_y$ are non-isomorphic but have the same class in $\Knull{X}$. It follows that the line bundles $\OX(x)$ and $\OX(y)$ have the same class in $\Knull{X}$ and then are isomorphic by Proposition~\ref{prop:LineBundleClass}. We thus assume $\OX(x)=\OX(y)$. Since $u=\al_x$ and $v=\al_y$ have non-isomorphic cokernels $S_x$ and $S_y$, the morphisms $u$ and $v$ are linearly independent over $k$. Hence the space $\Hom{\OX}{\OX(x(x)}$ has $k$-dimension at least two. From the exactness of the sequence $0\to \Hom{\OX}{\OX}\to \Hom{\OX}{\OX(x)}\to \Hom{\OX}{S_x}$ we infer that $\Hom{\OX}{\OX(x)}$ has exactly dimension two. Therefore
\begin{equation*}
  2-g = \euform{\OX}{\OX}+\euform{\OX}{S_x}=\euform{\OX}{\OX(x)},
\end{equation*}
implying that $\Ext{1}{\OX}{\OX(x)}$ has dimension $g$, the dimension of $\Ext{1}{\OX}{\OX}$. Invoking Serre duality it follows that $\Hom{\OX}{\tau\OX}$ and $\Hom{\OX(x)}{\tau\OX}$  have the same dimension $g$.

For each non-zero morphism $u$ from $\Hom{\OX}{\OX(x)}$ we obtain a non-split exact sequence
\begin{equation}\label{eqn;SimpleCoker}
  0\to \OX\stackrel{u}{\lra}\OX(x)\to S_u\to 0,
\end{equation}
where $S_u$ has finite length and degree one and hence is a simple sheaf. Using that $\Ext{1}{S_u}{\OX}$ has dimension one, it follows that two non-zero members $u$ and $v$ of $\Hom{\OX}{\OX(x)}$ have isomorphic cokernels if and only if $ku=kv$. Since $\Hom{\OX}{\OX(x)}$ has dimension two we obtain infinitely many pairwise non-isomorphic simple sheaves of the form $S_u$, $0\neq u\in\Hom{\OX}{\OX(x)}$.

Assume now that $g\geq1$ and hence $\Hom{\OX}{\tau\OX}$ is non-zero. Fixing an embedding $j:\OX\hookrightarrow\tau\OX$ and applying $\Hom{\OX}{-}$ to the sequence \eqref{eqn;SimpleCoker} we see that $u$ induces an epimorphism $\Ext{1}{\OX}{\OX}\to\Ext{1}{\OX}{\OX(x)}$ which is an isomorphism by equality of dimensions. By Serre duality this yields a bijection $\Hom{\OX(x)}{\tau \OX}\to \Hom{\OX}{\tau\OX}, h\mapsto h\circ u$. In particular, the inclusion $j:\OX\to \tau\OX$ factors into a product $j=h\circ u$ of two monomorphisms, implying that each simple module of the form $S_u$, $0\neq u\in\Hom{\OX}{\OX(x)}$, embeds into the cokernel $C$ of $j$. Since $C$ has finite length and there are infinitely many non-isomorphic simples of the form $S_u$, this is impossible, thus contradicting our assumption $g>0$.
\end{proof}

\subsection{Shift action associated to a point} \label{sect:ShiftAction}
For $x\in X$ let $S_x$ be the simple sheaf concentrated in $x$. For each vector bundle $E$ on $X$ we consider the $S_x$-universal extension
\begin{equation}\label{eqn:UniversalExtension}
  0\to E\up{\alpha_E} E(x) \to \Ext{1}{S_x}{E}\otimes_k S_x \to 0
\end{equation}
which is uniquely defined up to isomorphism. This defines a self-equivalence $\si_x:\vect{X}\to\vect{X}$ which extends to a self-equivalence of $\coh{X}$, also denoted $\si_x$ which, restricted to $\cohnull{X}$, is (isomorphic to) the identity on $\cohnull{X}$. For details we refer to \cite[Section 10.3]{Lenzing:2007}, see also \cite[Section 3]{Lenzing:Pena:1999}.

\begin{proposition}\label{prop:DivisorGroup}
The \emph{divisor group} $\Div{X}$, defined as the free abelian group on $X$, acts as group of automorphisms on the category $\coh{X}$ by $(\sum_{x\in X}{n_x}x)(E)=\left(\prod_{x\in X}{\si_x^{d_x}}\right)(E)=:E(D)$ if $D=\sum_{x\in X}d_x x$. Moreover the following holds
\begin{enumerate}
  \item For each line bundle $L$ and coherent sheaf of finite length $U$ there is --- up to isomorphism --- a unique line bundle $L(U)$ such that $[L(U)]=[L]+[U]$ holds in $\Knull{X}$.
  \item The action of the divisor group $\Div{X}$ on the category $\Ll(X)$ of line bundles is transitive, up to isomorphism.
\end{enumerate}
\end{proposition}
\begin{proof}
Due to the Hom-Ext-orthogonality for simple sheaves $S_x$, $S_y$ with $x\neq y$, the equivalences $\si_x$ and $\si_y$ always commute. This immediately implies the first assertion.

Invoking a composition series for the finite length sheaf $U$, the class $[U]$ is a (finite) linear combination $\sum_{x\in X}{d_x[S_x]}$ of classes of simple sheaves. Hence $L(U):=(\prod_{x\in X}\si_x^{d_x})(L)$ is a line bundle with class $[L]+[U]$. By Proposition~\ref{prop:LineBundleClass} $L(U)$ is uniquely determined by this property.

For the last assertion we consider line bundles $L_0$, $L$ and a point $x$. For large $n$ the Euler product $\euform{L_0}{L(nx)}$ is $>0$, yielding a non-zero morphism $L_0\to L(nx)$ and then a short exact sequence $0\to L_0\to L(nx)\to V\to 0$, where the cokernel-term $V$ has rank zero, and thus finite length. This implies $L(nx)=L_0(V)$ showing that $L$ and $L_0$ are in the same $\Div{X}$-orbit, proving the last claim.
\end{proof}
\begin{corollary} \label{cor:Knullprime}
Let $\Knullprime{X}$ denote the subgroup of $\Knull{X}$ generated by the classes $[S_x]$ of all simple sheaves $S_x$, $x\in X$. Then $\Knull{X}=\ZZ[\OX]\oplus\Knullprime{X}$.
\end{corollary}
\begin{proof}
The subgroups $\ZZ[\OX]$ and $\Knullprime{X}$ have zero intersection. To show that their sum equals $\Knull{X}$, it suffices by  Theorem~\ref{thm:CohX} to show  that the class of each line bundle belongs to $H$. This is an immediate consequence of Proposition~\ref{prop:DivisorGroup}.
\end{proof}
\subsection{The divisor sequence}\label{sect:DivisorSequence}
Let $X$ be a smooth projective curve. The next result is useful to treat the divisor theory of $X$.
\begin{proposition} \label{prop:coker}
Let $f:L\to\bar{L}$ be a nonzero morphism between line bundles. Then the cokernel of $f$ is a finite direct sum
$$
\coker{f}=\bigoplus_{x\in X}U_x,
$$
where for each $x\in X$ the object $U_x$ is either zero or else is an indecomposable finite length sheaf concentrated in $x$.

Further, as a member of $\Hom{L}{\bar{L}}$, $f$ is determined by the class $[\coker{f}]_0$ in
$\Knull\Hhnull$ up to a member of the multiplicative group $k^*$ of $k$.
\end{proposition}

\begin{proof} Assume that $U_y$ is decomposable for some $y\in X$. Then
$U_y$ maps onto $S_y\oplus S_y$, yielding an epimorphism $\bar{L}\twoheadrightarrow S_y\oplus S_y$.
Applying $\Hom{-}{S_y}$ thus yields a monomorphism
$k\oplus k=\Hom{S_y^2}{S_y}\to \Hom{\bar{L}}{S_y}=k$, contradiction. This
proves the first assertion.

Invoking this first result, the class $[\coker{f}]_0$ determines
$U=\coker{f}$ up to isomorphism. We prove the second claim by
induction on the length $\ell$ of $U$. For $\ell=0$ nothing is to
prove. If $\ell>0$ we choose a simple subobject, say $S_x$, of $U$
and analyze the pull-back diagram
$$
\def\bU{\bar{U}}
\def\bL{\bar{L}}
\xymatrix{
       &                       &        0            &        0        &          \\
       &                       &\bU\ar[u]\ar@{=}[r]  &      \bU\ar[u]  &          \\
0\ar[r]&L\ar[r]^f              &\bL\ar[u]\ar[r]      &U\ar[u]\ar[r]    &  0       \\
0\ar[r]&L\ar[r]_{f_x}\ar@{=}[u]&L_x\ar[r]\ar[u]_{g_x}&S_x\ar[u]_j\ar[r]&  0       \\
       &                       &0\ar[u]              &0\ar[u]          &
       }
$$
of $\eta$ along inclusion $j:S_x\hookrightarrow U$. Since $\Ext{1}{S_x}{L}=k$,
the line bundle $L_x$ is unique up to isomorphism and $f_x$, as a member of $\Hom{L}{L_x}$, is unique up to a scalar from $k^*$. By induction hypothesis, the same holds for $g_x$. The claim
follows.
\end{proof}

The \emph{function field} $\funct{X}$ of $X$ can be viewed as the endomorphism ring of the structure sheaf $\OX$ (or of any line bundle $L$) in the Serre quotient $\coh{X}/\cohnull{X}$, where the latter category then can be identified with the category $\mmod{\funct{X}}$ of finite dimensional $\funct{X}$-vector spaces. By Quillen's localization sequence, see \cite{Quillen:1973} the `exact sequence' of abelian categories \begin{equation*}
  0\to \cohnull{X}\to \coh{X}\to \mmod{\funct{X}}\to 0
\end{equation*}
gives rise to a long exact sequence of higher K-groups, which in a truncated form leads to the isomorphic exact sequences below
\begin{equation}\label{eqn:DivSequence}
\xymatrix@R12pt@C12pt{
0\ar[r]& k^*\ar@{=}[d] \ar[r] &\mathrm{K}_1(\mmod{\funct{X}}) \ar@{=}[d] \ar[r]&\Knull{\cohnull{X}}\ar@{=}[d]\ar[r]&\Knull{X}\ar@{=}[d]\ar[r]&\Knull{\mmod{\funct{X}}}\ar@{=}[d]\ar[r] &0\\
0\ar[r]& k^*\ar[r]  &\funct{X}^*\ar[r]^{\mathrm{div}}&\Div{X}\ar[r]^{\mathrm{nat}}&\Knull{X}\ar[r]^{\mathrm{rk}}&\ZZ\ar[r]&0
}
\end{equation}
where specifically the lower sequence is called the \emph{divisor sequence for $X$}.

\begin{theorem}\label{thm:DivisorSequence}
The divisor sequence for $X$ is exact.
\end{theorem}

\begin{proof}
We put $\Hh=\coh{X}$, $\Hhnull=\cohnull{X}$,  $\tHh=\Hh/\Hhnull$ and $L_0=\OX$.

The natural homomorphism $\End{\Hh}{L_0}\to \End\tHh{L_0}$, $f\mapsto \tf$, yields an interpretation of the embedding $k\hookrightarrow\funct{X}$ that we are going to use later on. In particular the sequence (\ref{eqn:DivSequence}) is exact at $k^*$.

The sequence is exact at $\funct{X}^*$: By construction $\divi{f}=0$ holds for each $f\in \funct{X}$ represented by a non-zero member, that is, an isomorphism in $\End\Hh{L_0}$. Assume conversely that $f$ is represented by $f':L' \to L_0$ with $0\neq L'\subset L_0$ and that we have
$$
0 = \divi{f}=[\coker{f'}]_0-[\coker{j}]_0,
$$
where $j:L'\hookrightarrow L_0$ is the inclusion. By Proposition~\ref{prop:coker} we conclude that there exists $\al\in k^*$ yielding a commutative diagram
$$
\xymatrix{
L_0 \ar[r]^\al & L_0\\
L' \ar[u]^j\ar[ru]_{f'}
}
$$
and $f=\al$ follows.

It is clear that composition $\mathrm{nat}\circ\mathrm{div}$ equals zero. Assume now that $a$ is in $\Knull\Hhnull$, then $a=[V]_0-[U]_0$ for suitable objects $U$ and $V$ from $\Hhnull$. We then get short exact sequences $0\to L' \up{s} L_0 \to U \to 0$ and $0\to L'' \up{a}L_0\to V \to 0$ and such that $[L']=[L'']$ in $\Knull\Hh$. Since line bundles are determined by their class in $\Knull{X}$ we may assume that $L'= L''$. Putting $f=as^{-1}$ we obtain $f\in \funct{X}$ with $\divi{f}=[V]_0-[U]_0$, hence the exactness of the divisor sequence at $\Knull\Hhnull$.

Since the rank vanishes on objects of $\Hhnull$ composition $\mathrm{rk}\circ\mathrm{nat}$ is zero. Conversely assume that $u=[X]-[Y]$ has rank zero such that $X$ and $Y$ have the same rank $r$. By Corollary~\ref{cor:Knullprime} $[X]=r[L_0]+x$ and $[Y]=r[L_0]+y$ where $x$ and $y$ hence also $u=x-y$ belong to the image of $\mathrm{nat}$. This finishes the proof of the theorem.
\end{proof}

In particular, the image of the natural map $\Knull{\cohnull{X}}\to \Knull{\coh{X}}$, $[U]_0\mapsto [U]$, equals the cokernel of the divisor map $\divi{}: \funct{X}^*\to \Knull{\cohnull{X}}=\Div{X}$, the \emph{divisor class group} $\Cl{X}$ of $X$.

\begin{corollary}
  The map $\Cl{X}\oplus\ZZ[\OX]\to \Knull{\coh{X}}$, $([D],n[\OX]\mapsto [D]+n[\OX]$, yields an isomorphism $\Knull{\coh{X}}=\Cl{X}\oplus\ZZ[\OX]$.~\qed
\end{corollary}
\begin{remark}
Let $\Pic{X}$ denote the \emph{Picard group} of $X$, that is, the group of isomorphism classes of line bundles on $X$ with respect to the tensor product. The action $\Div{X}\times \Pic{X}\to \Pic{X}$, $D\mapsto \OX(D)$, from Proposition~\ref{prop:DivisorGroup} induces an isomorphism between the class group $\Cl{X}$ and the Picard group. The surjective degree function $\Pic{X}\to \ZZ$ further induces a splitting $\Pic{X}=\Picnull{X}\oplus \ZZ$. It is known that $\Picnull{X}$ is isomorphic to the Jacobian variety of
$X$. Assume $g\geq1$. For $k=\CC$, the field of complex numbers,
it is classical that the Jacobian variety is isomorphic to
$(\CC/(\ZZ\times\ZZ))^{g}$~\cite{Hartshorne:1977}, but also for an
arbitrary algebraically closed field the Jacobian variety is a
(nontrivial) divisible abelian group, see~\cite{Mumford:2008},
which makes it impossible that $\Knull{X}$ is either free or
finitely generated.
\end{remark}

\section{Coherent sheaves on weighted projective curves}

\subsection{The category of $p$-cycles}

Let $\mathcal{H}$ be a category of coherent sheaves on a smooth projective
curve. Fixing a point $x$ of $X$ and an integer $%
p\geq 2$ we are going to form the category $\mathcal{\bar{H}}$ of $p$-cycles
in $x$ which will be viewed as the category of coherent sheaves on the weighted curve $%
\XX$, having $X$ as underlying smooth projective curve and $x$ as its only weighted point (of weight $p$). Intuitively speaking the effect of the following construction is to form a $p$-th root
of the natural transformation ${x}_{E}:E\rightarrow E(x)$ corresponding to
the point $x$, relating to the construction of the
Riemann surface of the $p$-th root function. This approach to weighted projective curves
is particularly suitable for a K-theoretic analysis. Below we reproduce definition and main properties of $p$-cycles from \cite{Lenzing:1998}.

\begin{definition}\em A $p$-\emph{cycle} $E$ \emph{concentrated in} $x$ is a diagram
\[
\cdots   \rightarrow  {E}_{n}\stackrel {{x}_{n}}{ \longrightarrow  }{E}_{n + 1}
\stackrel{x_{n+1}}\longrightarrow{E}_{n + 2}\stackrel {{x}_{n + 2}}{ \longrightarrow  }\cdots   \longrightarrow
{E}_{n + p}\stackrel {{x}_{n + p}}{ \longrightarrow  }\cdots
\]
\noindent
which is $p$-{\em periodic} in the sense that ${E}_{n + p}  = {E}_{n} (x), $ ${x}_{n
+ p}  = {x}_{n} (x)$ and moreover ${x}_{n + p - 1} \circ \cdots  \circ {x}_{n}
 = {x}_{{E}_{n}}$ holds for each integer $n$.

A morphism $u : E \to  F$ of $p$-cycles concentrated in $x$ is a sequence of morphisms
${u}_{n}  : {E}_{n}  \to  {F}_{n}$ which is $p$-\emph{periodic}, i.e.\ satisfies
${u}_{n + p}  = {u}_{n}$ for each $ n$, and such that each diagram
\[
\begin{array}{ccc}
{E}_{n} &  \stackrel{x_n}{\longrightarrow}   & {E}_{n + 1}\\
{\scriptstyle {u}_{n}} \downarrow  &  & \downarrow {\scriptstyle {u}_{n + 1}}\\
{F}_{n} & \stackrel {{x}_{n}}{ \longrightarrow  } & {F}_{n + 1}\end{array}
\]
\noindent
commutes. We denote $p$-cycles in the form ${E}_{0}\stackrel {{x}_{0}}{
\longrightarrow  }{E}_{1}\stackrel {{x}_{1}}{ \longrightarrow  }\cdots   \to
{E}_{p - 1}\stackrel {{x}_{p - 1}}{ \longrightarrow  }{E}_{0} (x)$   and the
category of all  $p$-cycles concentrated in $x$ by $\overline{\mathcal{H} }  =
\mathcal{H} \scriptstyle \left ({\begin{array}{c}
p\\
x\end{array}}\right )$.
\end{definition}

Obviously $\overline{\mathcal{H}}$ is an abelian category, where exactness
and formation of kernels and cokernels has a pointwise interpretation.
Moreover we have a full exact embedding
\[
j:\mathcal{H}\hookrightarrow \overline{\mathcal{H}},\qquad \;E\mapsto \bar{E}%
=E=\cdots =E\stackrel{{x}_{E}}{\longrightarrow }E(x).
\]
\noindent We therefore identify $\mathcal{H}$ with the resulting exact
subcategory of $\overline{\mathcal{H}}$. We note that inclusion $j:\mathcal{H%
}\to \overline{\mathcal{H}}$ has a left adjoint $\ell $ and a right
adjoint $r$ which are both exact functors and are given by

\begin{equation}
\ell \left( E_{0}\stackrel{x_{0}}{\to }E_{1}\stackrel{x_{1}}{%
\to }\cdots \to E_{p-1}\stackrel{x_{p-1}}{\to }%
E_{0}(x)\right) =E_{p-1}
\end{equation}
and
\begin{equation}
r\left( E_{0}\stackrel{x_{0}}{\to }E_{1}\stackrel{x_{1}}{\to
}\cdots \to E_{p-1}\stackrel{x_{p-1}}{\to }E_{0}(x)\right)
=E_{0}.
\end{equation}

\begin{lemma}The category $\overline{\mathcal{H} }$ is connected, abelian and
noetherian. The simple objects (sheaves) of $\overline{\mathcal{H} }$ occur in two
types:

\begin{enumerate}
\item
the simple sheaves $S_y$ of $\mathcal{H} $ which are concentrated in a point $y\neq x$.

\item
the $p$ simple sheaves
\[
\begin{array}{ccc}
{S}_{1}  :  &  & 0 \to  0 \to  \cdots   \to  0
\to  S_x \to  0\\
{S}_{2} : &  & 0 \to  0 \to  \cdots   \to  S_x
\to  0 \to  0\\
\cdots  &  & \cdots \\
{S}_{p - 1} : &  & 0 \to  S_x \to  \cdots   \to  0
\to  0 \to  0\\
{S}_{p} : &  & S \to  0 \to  \cdots   \to  0
\to  0 \to  S_x(x),\end{array}
\]
\end{enumerate}
said to be concentrated in $x$.

Each ${S}_{i}$ is exceptional and $\Ext{1}{S_i}{S_{i+1}}=k$ for each $i\in\ZZ_p$.

If $\mathcal{S}  =  \left \{{{S}_{1} , \ldots  , {S}_{p - 1} |S \mbox{ simple
in } \mathcal{H}  \mbox{ and concentrated in } x}\right \}$, then the
extension closure $\wt{\mathcal{S}}$ of $\mathcal{S} $ is
localizing in $\overline{\mathcal{H} }$. Moreover

\renewcommand{\labelenumi}{\alph{enumi}.}
\begin{enumerate}
\item
the quotient category $\overline{\mathcal{H} } /\wt{\mathcal{S}}
\cong  \mathcal{H} $ is isomorphic to $\mathcal{H} $, the isomorphism induced
by $r :  \overline{\mathcal{H} }  \to  \mathcal{H} $.

\item
the right perpendicular category ${\mathcal{S} }^{\perp }$ formed in
$\overline{\mathcal{H} }$ is equivalent to $\mathcal{H} $.
\end{enumerate}
\end{lemma}

\begin{proof} Abelianness and noetherianness are obvious by pointwise
consideration. It is straightforward from the construction and from the
properties of $\mathcal{H}$ that the category ${\mathcal{S}}^{\perp }$ right
perpendicular to $\mathcal{S}$ consists of exactly those $p$-cycles $E_{0}%
\stackrel{x_{0}}{\to }E_{1}\stackrel{x_{1}}{\to }\cdots
\to E_{p-1}\stackrel{x_{p-1}}{\to }E_{0}(x)$ such that $%
x_{0},\ldots ,x_{p-2}$ are isomorphisms, hence --- up to isomorphism ---
agree with the objects from $\mathcal{H}$. The remaining properties are
straightforward to check.
\end{proof}

Note that the formation of categories of $p$-cycles may be iterated. Let therefore $x_1,x_2,\ldots,x_t$ be pairwise distinct points of $X$ and $p_1,p_2,\ldots,p_t$ a corresponding sequence of weights, the iterated $p$-cycle construction leads to a category $\overline{\Hh}=\Hh\begin{pmatrix}
                p_1,p_2,\ldots,p_t \\
                x_1,x_2,\ldots,x_t
              \end{pmatrix}$
that we view as category of coherent sheaves $\coh{\XX}$ for the weighted projective curve given by the above weight data.
Comparison with the categories of~\cite{Reiten:VandenBergh:2002} yields the following:

\begin{theorem}The category $\overline{\Hh}=\coh{\XX}$ is again a hereditary noetherian $k$-category with Serre duality in the form $\Ext{1}{A}{B}=\dual{\Hom{B}{\tau A}}$. Moreover, the following properties hold:
\begin{enumerate}
  \item For each $x$ in $X$ there is a unique simple object (sheaf) $S_x$ concentrated in $x$ with the property $\Hom{\OX}{S_x}=k$.
  \item For $i=1,\ldots,t$ we put $S_i=S_{x_i}$. Then the simple sheaves concentrated in $x_i$ are exceptional and form a $\tau$-orbit $(\tau^jS_i)_{j\in\ZZ_{p_i}}$ of period $p_i$. The remaining simples satisfy $\tau S_x=S_x$.
  \item The exceptional simple sheaves, different from $S_1,S_2,\ldots,S_t$ can be arranged to form an exceptional sequence $\Ee$ such that the right perpendicular category $\rperp{\Ee}$, formed in $\overline{\Hh}$, equals $\Hh=\coh{X}$.\qed
\end{enumerate}
\end{theorem}
\begin{corollary}
The category $\Hh$ is a full, exact subcategory of $\overline{\Hh}$, and inclusion induces an equivalence of Serre quotients $\Hh/\Hhnull=\overline{\Hh}/\overline{\Hh}_0$. Accordingly, $\Hh$ and $\overline{\Hh}$ share the same function field $K$ and inclusion from $\Hh$ to $\overline{\Hh}$ is rank-preserving.\qed
\end{corollary}
In particular, the structure sheaf $\OX$ of $\Hh$ is a line bundle in $\overline{\Hh}$ which we use as \emph{structure sheaf} for $\overline{\Hh}$: Notation: $\OXX$. We further put $\coh{\XX}=\overline{H}$ and $\Knull{\XX}=\Knull{\coh{\XX}}$.

By \cite{Geigle:Lenzing:1991} and making use of Proposition~\ref{prop:RClass} we get the following immediate consequence of the proposition.
\begin{corollary}
The Grothendieck group $\Knull{\XX}$ decomposes into a direct sum
\begin{equation*}
\Knull{\XX}=\Knull{X}\oplus\left(\bigoplus_{E\in \Ee}\ZZ[E]\right).
\end{equation*}
In particular, $\Knull{\XX}=\Knull{X}\oplus\ZZ^n$ with $n=\sum_{i=1}^{t}(p_i-1)$. Moreover, $\euform{[E]}{[H]}=0$ whenever $E$ belongs to $\Ee$ and $H$ belongs to $\coh{X}$.\qed
\end{corollary}

\subsection{The reduced (or numerical) Grothendieck group}

Let $\XX=(X,w)$ be a weighted projective curve. When speaking of the Grothendieck group $\Knull{\XX}=\rKnull{\coh\XX}$ we always assume that the Euler form is part of its structure. Factoring out the kernel of the Euler form, we obtain the reduced (or numerical) Grothendieck group $\rKnull{\XX}$, equipped with a non-degenerate Euler form.

\begin{lemma}
The kernel of the Euler form for $\Knull{\XX}$ is generated by the elements $u_x-u_y$, $x,y\in X$, where $u_x=\sum_{j\in \ZZ_{w(x)}}[\tau^jS_x]$ for the unique simple sheaf $S_x$ concentrated in $x$ with $\Hom{\OX}{S_x}=k$.~\qed
\end{lemma}
 To analyze properties of the Grothendieck group $\Knull{\XX}$, or the \emph{reduced Grothendieck group} $\rKnull{\XX}=\Knull{\XX}/N$, the following proposition, going back to H{\"u}bner \cite{Huebner:1996}, compare~\cite{Lenzing:2007}, is quite useful. An exceptional sequence $E_1,E_2,\ldots,E_n$ in $\coh{\XX}$ consists of exceptional sheaves such that additionally $\Ext{j}{E_h}{E_i}=0$ holds for $h>i$ and for $j=0,1$.
\begin{proposition} \label{prop:RClass}
Each exceptional sheaf $E$ of $\coh{\XX}$ is determined by its class
$\rcl{E}$ in the reduced Grothendieck group $\rKnull{\Hh}$. In particular, $E$ is determined by its class in $\Knull{\XX}$. Moreover, an exceptional sequence in $\coh{X}$ gives rise to a linearly independent sequence $\rcl{E_1}, \rcl{E_2},\ldots,\rcl{E_n}$ in $\rKnull{\XX}$ (resp.\ $[E_1], [E_2],\ldots,[E_n]$ in $\Knull{\XX}$, generating a direct summand in $\rKnull{\XX}$ (resp.\ $\Knull{\XX}$).
\end{proposition}
\begin{proof} Let $E$ and $F$ be exceptional sheaves with
$\rcl{E}=\rcl{F}$. Then
$\euform{\rcl{E}}{\rcl{F}}=\euform{\rcl{E}}{\rcl{E}}=1$, and
hence there exists a non-zero morphism $f:E\to F$ whose kernel and
image we denote by $E'$ and $F'$, respectively. We claim that $f$ is
an isomorphism and assume, for contradiction, that $E'$ is non-zero.
The assumption implies that $\Hom{F'}{E}=0$: Otherwise there
exists a nonzero morphism $g:F'\to E$ and the composition $E\twoheadrightarrow
F'\up{g} E$ yields a non-trivial endomorphism, hence an
automorphism, of $E$. It then follows that the morphism $E\twoheadrightarrow F'$,
induced by $f$, is an isomorphism, contradicting $E'\neq0$. We have
shown that the assumption $E'\neq0$ implies $\Hom{F'}{E}=0$.
Further, since $\coh{\XX}$ is hereditary the embedding $F'\hookrightarrow F$ induces
an epimorphism $0=\Ext{1}{F}{F}\twoheadrightarrow \Ext{1}{F'}{F}$ implying
$\Ext{1}{F'}{F}=0$. We thus obtain
$\euform{\rcl{F'}}{\rcl{E}}=-\dim{}{\Ext{1}{F'}{E}}\leq0$ and
$\euform{\rcl{F'}}{\rcl{F}}=\dim{}{\Hom{F'}{F}}\geq0$. Since
the classes of $E$ and $F$ in $\rKnull{\XX}$ agree, we get
$0=\euform{\rcl{F'}}{\rcl{E}}=
\euform{\rcl{F'}}{\rcl{F}}=\dim{}{\Hom{F'}{F}}$, hence
$\Hom{F'}{F}=0$, contradicting $f\neq0$. We have thus shown that
there exists a monomorphism $f:E\to F$. Similarly, there exists a
monomorphism $g:F\to E$ yielding non-zero endomorphisms, hence
automorphisms, $f\circ g$ and $g\circ f$ of $F$ resp.\ $E$.
Therefore $f$ and $g$ are isomorphisms and the first claim follows. The last
assertion uses that $\euform{E_i}{E_i}=1$ for each $i$ and $\euform{E_j}{E_i}=0$ for $j>i$.
\end{proof}

\subsection{Attaching tubes}
An abelian group $V$, equipped with a (usually non-symmetric) bilinear form $\euform{-}{-}$ and a $\ZZ$-linear automorphism $\tau$ is called a \emph{bilinear group} if the identity $\euform{y}{x}=-\euform{x}{\tau y}$ holds for all members $x,y$ of $V$. A bilinear group $V$ is called a \emph{bilinear lattice} if $V$ is a free abelian group of finite rank and additionally the bilinear form is non-degenerate. For weighted projective lines, the process of attaching tubes from \cite{Lenzing:1996}, that is described below, has been the key technique to understand the structure of the Grothendieck group.

A $p$-tube is a pair $(T,S)$ where $T$ is a bilinear group equipped with an automorphism $\tau$ of finite period $p$, $s\in T$ satisfies $\euform{s}{s}=1$ and, moreover, the $\tau$-orbit $(\tau^js)_{s\in\ZZ_p}$ is a $\ZZ$-basis of $T$.
We further request
 \begin{equation} \label{eqn:tube}
 \euform{\tau^i s}{\tau^j s}=\begin{cases} 1 & \mbox{if $j=i$ in $\ZZ_p$} \\
                        -1 & \mbox{if $j=i+1$ in $\ZZ_p$} \\
                         0                      & \mbox{else,}
                         \end{cases}
 \end{equation}
We note that
$s_0=\sum_{j\in\ZZ_p}s_j$ lies in the kernel of $T$, that is $\euform{s_0}{-}$  so $T$ is not a bilinear lattice. Moreover, $\euform{s_0}{s_0}=0$.

Now let $V$ denote the bilinear group $\Knull{X}$ for a smooth projective curve $X$ and let $S$ be a simple sheaf on $X$. We are going to define a bilinear group $\bar{V}$ by attaching a $p$-tube $T_p$ at $[S]$. First we define a bilinear form on $V\oplus T$ by putting
\begin{enumerate}
  \item $\euform{v_1}{v_2}=\euform{v_1}{v_2}_V$ for members $v_1,v_2$ from $V$ and $\euform{u_1}{u_2}=\euform{u_1}{u_2}_T$ for members $u_1,u_2$ from $T$,
  \item \begin{equation*}
          -\euform{\tau^{j-1}s}{v}=\euform{v}{\tau^js}=\begin{cases}
                                                       \euform{v}{[S]}_V, & \mbox{if } j=1 \text{ in } \ZZ_p, v\in V\\
                                                       0, & \mbox{otherwise}.                                                     \end{cases}
        \end{equation*}

\end{enumerate}
As is easily checked, the subgroup $U$ generated by $\sum_{j\in\ZZ_p}{\tau^js}-[S]$ belongs to the kernel of the bilinear form on $V\oplus T$, thus yielding a bilinear group $\bar{V}=V\oplus T/U$. By definition, $\bar{V}$ is the \emph{bilinear group obtained by attaching a $p$-tube} in $[S]$.
Similarly, starting with the bilinear lattice $\rKnull{X}$ and the class $\rcl{S}$ of a simple sheaf on $X$ we obtain by attaching a $p$-tube to $\rcl{S}$ a bilinear lattice $\bar{V}$.

For a given sequence of weights $p_1,p_2,\ldots,p_t$ and pairwise non-isomorphic simple sheaves $S_1,S_2,\ldots, S_t$ the above construction extends to yield an extension of $\Knull{X}$, respectively $\rKnull{X}$, by simultaneously attaching a $p_i$-tube in $[S_i]$, respectively $\rcl{S_i}$, for $i=1,2,\ldots,t$.
\begin{theorem} \label{thm:InsertionOfTubes}
Let $\XX$ be a weighted projective curve, with weights $p_1,p_2,\ldots,p_t$ put in the points $x_1,x_2,\ldots,x_t$ of the smooth projective curve $X$ underlying $\XX$.  Then Grothendieck group $\Knull{\XX}$ of $\coh{\XX}$ arises from $\Knull{X}$  by inserting a $p_i$-tube $T_i$ in the class $[S_i]$ of the simple sheaf $S_i$ concentrated in $x_i$, for each $i=1,2,\ldots,t$.
\end{theorem}

\begin{proof}
For each $i=1,2,\ldots t$ let $S_i$ denote the unique simple sheaf on $\XX$ concentrated in $x_i$ with the property $\Hom{\OXX}{S_i}=k$. Further let $\Ee$ denote the system of all remaining exceptional simple sheaves ordered into an exceptional sequence. We thus obtain
a semi-orthogonal decomposition $\wt{\coh{X}, \Ee}$, yielding a decomposition of $\Knull{\XX}$ into $\Knull{X}$ and the subgroup of the classes spanned by the linearly independent system of classes $E$ with $E$ from $\Ee$. From this the proof of the claim is straightforward.
\end{proof}

\begin{corollary}
The reduced Grothendieck group $\rKnull{\XX}$ is the bilinear lattice arising from $\rKnull{X}$ by inserting $p_i$-tubes in the classes $\rcl{S_i}$, $i=1,\ldots,t$.
Moreover, the Coxeter transformation, induced by $\tau$ on $\rKnull{\XX}$ yields as characteristic polynomial the Coxeter polynomial $(x-1)^2\prod_{i=1}^{t}v_{p_i}$, where $v_n=(x^n-1)/(x-1)$.
\end{corollary}
\begin{proof}
For the last assertion we use that, according to \cite{Lenzing:1996}, each attachment of a $p$-tube in $\rcl{S_x}$ for a non-weighted point $x$ yields an additional factor $v_p$ for the Coxeter polynomial.
\end{proof}
We note that the weight sequence of $\XX$ can be recovered from its Coxeter polynomial. However, the genus of the underlying curve $X$ does not influence the Coxeter polynomial.
With the above notations, we put $a=\rcl{\OXX}$, $s_i=\rcl{S_i}$, $s_0=\rcl{S_x}$ for any ordinary simple sheaf from $\coh{X}$. Then the elements $a$, $(\tau^js_i)_{j\in\ZZ_{p_i}}$, and $s_0$ form a generating system for $\rKnull{\XX}$.

\subsection{Orbifold Euler characteristic and weighted Riemann-Roch}
We recall that for a smooth projective curve $X$ the \emph{Euler characteristic} (or \emph{Euler number}) $\chi_X$ is given by the expression $\chi_X=2\euform{\OX}{\OX}=2(1-g_X)$, where $g_X$ is the genus of $X$, defined as the $k$-dimension of $\Ext{1}{\OX}{\OX}$. Assume now that $\XX=X\wt{p_1,p_2,\ldots,p_t}$ is a weighted projective curve with structure sheaf $\OXX$. Then $\Ext{1}{\OXX}{\OXX}$ has dimension $g_X$, so the \emph{homological genus} does not reflect the weights. To define an Euler characteristic $\chi_\XX$ of $\XX$ that keeps notice of the weights, we use the \emph{weighted Euler form}, defined as an average of $\bar{p}=lcm(p_1,p_2,\ldots,p_t)$ twisted Euler forms:
\begin{equation}\label{eqn:AveragedEulerForm}
  \aveuform{x}{y}=\frac{1}{\bar{p}}\left( \sum_{j=0}^{\bar{p}-1}\euform{\tau^j x}{y}\right).
\end{equation}
Here, $\tau=\tau_\XX$ denotes the Auslander-Reiten translation for $\coh\XX$.
We then define the \emph{orbifold Euler characteristic} of $\XX$ as $\chi_\XX=2\aveuform{\OXX}{\OXX}$. Clearly, for unweighted smooth projective curves the orbifold Euler characteristic agrees with the ordinary Euler characteristic.

\begin{theorem}\label{thm:OrbifoldEulerChar}
Let $\XX=X\wt{p_1,p_2,\ldots,p_t}$ be a weighted projective curve with underlying smooth projective curve $X$. Then the orbifold Euler characteristic of $\XX$ is given by the expression
\begin{equation}\label{eqn:OrbifoldEulerChar}
  \chi_\XX=2\aveuform{\OXX}{\OXX}=\chi_X-\sum_{i=1}^{t}\left(1-\frac{1}{p_i}\right).
\end{equation}
\end{theorem}
\begin{proof}
We work in the reduced Grothendieck group $\rKnull{\XX}$. We put $a=\rcl{\OXX}$, $s_0=\rcl{S_x}$ where $S_x$ is the simple sheaf concentrated in a non-weighted point, $s_i=\rcl{S_i}$, where $S_i$ is the simple sheaf concentrated in $x_i$ with $\Hom{\OXX}{S_i}=k$ ($i=1,2,\ldots,t$). Then we have
\begin{equation}\label{eqn:Tau}
  \tau a -a = -\sum_{i=1}^{t}s_i+(t-\chi_X)s_0,
\end{equation}
where $\chi_X=2\euform{a}{a}=2(1-g)$ is the Euler characteristic of the underlying smooth projective curve $X$. The validity of formula \eqref{eqn:Tau} is easily checked by forming for both sides the Euler product with each member of the generating system $a,s_0,\tau^j s_i$, $i=1,2,\ldots,t$ and using non-degeneracy of the Euler form on $\rKnull{\XX}$. Then, applying $\sum_{j=0}^{\bar{p}-1}\tau^j$ to \eqref{eqn:Tau} implies
\begin{equation}\label{eqn:TauUpP}
  \tau^{\bar{p}}a-a=\bar{p}\,\delta\, s_0,\text{ where }\delta=\sum_{j=1}^{t}(1-\frac{1}{p_i})-\chi_X.
\end{equation}
Finally, we obtain
\begin{equation}
  2\aveuform{a}{a}+\delta = \aveuform{a}{a}+\aveuform{a}{\tau^{\bar{p}}a} =
  \sum_{j=0}^{\bar{p}-1}\euform{\tau^ja}{a}-\sum_{j=0}^{\bar{p}-1}\euform{\tau^{(\bar{p}-1)-j}a}{a}  =0,
\end{equation}
proving the claim.
\end{proof}
Also for a weighted projective curve $\XX$ we have two ``obvious'' linear forms \emph{rank} and \emph{degree} on the Grothendieck group $\Knull{\XX}=\Knull{\coh{\XX}}$ given on coherent sheaves by the expressions
\begin{equation}\label{eqn:WeightedRankDegree}
  \rk{A}=\aveuform{A}{S_x} \text{ and } \dg{A}=\aveuform{\OXX}{A}-\aveuform{\OXX}{\OXX}\rk{A}.
\end{equation}
Again this K-theoretic rank agrees with the sheaf-theoretic rank, defined for a coherent sheaf $A$ as its length in the Serre quotient $\coh{\XX}/\cohnull{\XX}$.

\begin{theorem}[Riemann-Roch]
Assume $b,c$ are members of $\Knull{\XX}$, where $\XX=X\wt{p_1,p_2,\ldots,p_t}$ then
\begin{equation}\label{eqn:WeightedRR}
  \aveuform{b}{c}=\aveuform{\OX}{\OX}\cdot\rk{b}\cdot \rk{c}+\frac{1}{\bar{p}}\cdot\begin{vmatrix}
                                                        \rk{b} \ & \rk{c} \\
                                                        \dg{b}& \dg{c}
                                                       \end{vmatrix}.
\end{equation}
\end{theorem}
\begin{proof}
Passing to the reduced Grothendieck group $\rKnull{\XX}$ one checks the equality of both sides of \eqref{eqn:WeightedRR} on the generating system $a,s_0,\tau^js_i$, $i=1,2,\ldots,t$, $j\in\ZZ_{p_i}$.
\end{proof}
\begin{remark}
The factor $1/\bar{p}$ in front of the determinant in the Riemann-Roch formula is due to the fact that we have normalized the degree in order to be always an integer. One may argue in favour of a different normalization of the degree such that simple sheaves concentrated in ordinary points always get degree one and simple sheaves concentrated in a point of weight $p$ get degree $1/p$. With this normalization the factor $1/\bar{p}$ disappears.
\end{remark}

\subsection{Impact of the Euler characteristic}
In this section we collect a number of facts showing the importance of the Euler characteristic. The following trichotomy is known for weighted projective lines and has a proof along similar lines for weighted projective curves.
\begin{theorem}
Let $\XX=X\wt{p_1,p_2,\ldots,p_t}$ and let $\chi_\XX=\chi_X-\sum_{i=1}^{t}(1-1/p_i)$ be its orbifold Euler characteristic. We have the following three cases:
\begin{enumerate}
  \item If $\chi_\XX>0$, then $X=\PP^1$ and the weight sequence is one of $\wt{p,q}$ ($p,q\geq1$), $\wt{2,2,n}$ $(n\geq2)$, $\wt{2,3,3}$, $\wt{2,3,4}$ or $\wt{2,3,5}$. In these cases $\coh{\XX}$ has tame domestic representation type.
  \item If $\chi_\XX=0$, then $\XX$ is either $\PP^1\wt{2,3,6}$, $\PP^1\wt{2,4,4}$, $\PP^1\wt{3,3,3}$, $\PP^1\wt{2,2,2,2}$ or else $\XX$ is a non-weighted smooth elliptic curve. In these cases $\coh{\XX}$ has tame tubular representation type.
  \item In all remaining cases $\XX$ has negative Euler characteristic $\chi_\XX<0$ and wild representation type.\qed
\end{enumerate}
\end{theorem}

For the base field of complex numbers more specific assertions can be made, because then the trichotomy agrees with the classical division of plane geometry into  euclidean (elliptic), spherical (parabolic) and hyperbolic geometry, provided we exclude weighted projective lines $\PP^1(p,q)$ with weights $p\neq q$, $p,q\geq1$. For the arising concepts \emph{orbifold fundamental group}, \emph{universal orbifold cover} and for the proof we refer to \cite{Lenzing:2016b}.

\begin{theorem} \label{thm:Uniformization}
Let $\XX$ be a weighted projective curve over $\CC$. Then the fundamental group $\pi_1^{orb}(\XX)$ acts on the universal orbifold cover $\wtilde{\XX}$ as group of deck transformations (the members of $\Aut\XX$ commuting with $\pi:\wtilde\XX\to \XX$).  This action represents $\XX$ as orbifold quotient
  $$\XX={\wtilde{\XX}}/{\pi_1^{orb}(\XX)}.
  $$
Assuming, moreover, that $\XX$ is not isomorphic to any $\PP^1\wt{p_1,p_2}$ with $p_1\neq p_2$, we have the following trisection:
\begin{enumerate}
  \item \emph{spherical}: If $\chi_\XX>0$ then $\XX$ admits a K{\"a}hler metric of constant curvature $+1$. Further $\wtilde{\XX}=\PP^1$, $\pi_1^{orb}(\XX)$ is a finite polyhedral group and $\XX$ is one of $\PP^1\wt{n,n}$, $\PP^1\wt{2,2,n}$, or $\PP^1\wt{2,3,p}$ with $p=2,3,5$.
  \item \emph{parabolic}: If $\chi_\XX=0$ then $\XX$ admits a K{\"a}hler metric of constant curvature $0$. Further $\wtilde{\XX}=\CC$, and $\XX$ is either a smooth elliptic curve or a weighted projective line of tubular type $\wt{2,3,6}$, $\wt{2,4,4}$, $\wt{3,3,3}$ or $\wt{2,2,2,2}$.
  \item \emph{hyperbolic}: If $\chi_\XX<0$ then $\XX$ admits a K{\"a}hler metric of constant curvature $-1$, $\wtilde{\XX}=\HH$, and $\XX$ has hyperbolic type.\qed
\end{enumerate}
\end{theorem}
We only remark that the proof of the theorem is based on the general Riemann mapping theorem, that is, the uniformization theorem of Poincar{\'e} and Koebe~\cite{De:Saint-Gervais:2016} stating that a simply connected Riemann surface is holomorphically isomorphic to either $\PP^1$, $\CC$ or $\HH$, where $\HH$ denotes the open upper complex half-plane. It also uses that $\PP^1$, $\CC$ or $\HH$ have K{\"a}hler metrics with constant curvature $+1$, $0$ and $-1$, respectively. For further details we refer to \cite{Lenzing:2016b}.

Still assuming the base field of complex numbers, the Euler characteristic develops its full power with the Riemann-Hurwitz theorem implying, in particular, that the trisection --- discussed above --- is preserved under finite orbifold group quotients. If $G$ is a finite group of automorphisms of $\XX=(X,w)$, then there are only finitely many orbits $Gx$ in $X/G$ having a non-trivial stabilizer $G_x$ (which is necessarily cyclic). Putting $\bar{w}(Gx)=w(x)\cdot|G_x|$ defines $\XX/G=(X/G,\bar{w})$ as a weighted projective curve, called the \emph{orbifold quotient}, or just quotient, of $\XX$ by $G$. By definition the automorphism group $\Aut{\XX}$ of a weighted projective curve $\XX=(X,w)$ is the subgroup of all automorphisms of $X$ that commute with $w$.

\begin{theorem}[Riemann-Hurwitz]
Let $\XX$ be a weighted projective curve and $G$ a finite subgroup of the automorphism group $\Aut{\XX}$ of $\XX$. Then $\XX/G$ is again a weighted projective curve having Euler characteristic $\chi_{\XX/G}={\chi_\XX}/{|G|}$.\qed
\end{theorem}

\subsection{Shift action, weighted  divisor group and weighted Picard group}
An early instance of weighted divisor theory is \cite[Section~7]{Lenzing:1986} which --- in retrospect --- deals with the divisor theory of weighted projective lines of tame domestic representation type, that is, of those having positive Euler characteristic. We are going to extend this treatment to weighted projective curves, in general.

Let $(X,w)$ be a weighted projective curve. By definition, the \emph{weighted divisor group}, or simply divisor group, $\Div{\XX}$ of $\XX$ is defined as the free abelian group on $X$. We first show that $\Div{\XX}$  acts on the category $\coh{\XX}$ as a group of self-equivalences. The approach is similar to Section~\ref{sect:ShiftAction} but, due to the weights, a bit more sophisticated: For each point of $x$ let $S_x$ be the unique simple sheaf on $\XX$ concentrated in $x$ which satisfies $\Hom{\OXX}{S_x}=k$. If $p_x$ is the weight of $x$, then the simple sheaves, concentrated in $x$, form a $\tau$-orbit $(\tau^jS_x)_{j\in \ZZ_{p_x}}$ of cardinality $p_x$. For each vector bundle $E$ we form the universal extension
\begin{equation}
0 \to E \up{x_E} E(x) \to  E_x \to 0,  \text{ where }E_x=\bigoplus_{j\in\ZZ_{p_x}}\Ext{1}{ \tau{}^j S_x}{E}\otimes
\tau{}^j S_x.
\end{equation}
By construction $E_x$ is a semi-simple sheaf, that is, a finite direct sum of simple sheaves.
As in Section~\ref{sect:ShiftAction} the assignment $E\mapsto E(x)$ leads to a self-equivalence $\si_x:\vect{\XX}\to\vect{\XX}$ that extends to a self-equivalence of $\coh{\XX}$ which we denote by the same symbol $\si_x$, and which is called \emph{point-shift by} $x$. Notation: $\si_xE=E(x)$ for $E$ in $\coh{\XX}$.
\begin{remark}
When extended to the bounded derived category the point shift $\si_x$ is a special instance of what is termed a \emph{tubular mutation} in \cite{Meltzer:1997} or otherwise  a \emph{spherical twist} or \emph{Seidel-Thomas twist} \cite{Seidel:Thomas:2001}.
\end{remark}
Specializing to a line bundle $L$, the point-shift construction leads to an exact sequence $0\to L \to L(x) \to L_x \to 0$, where $L_x$ is the unique simple sheaf concentrated in $x$ which satisfies $\Ext{1}{L_x}{L}=k$. We recall that all the other simple sheaves $\tau^j L_x$, $0\neq j \in \ZZ_{p_x}$, that are concentrated in $x$, satisfy $\Ext{1}{\tau^jL_x}{L}=0$.

In this section, the next section and the appendix we are going to encounter four conceptually different incarnations of the \emph{Picard group} $\Pic{\XX}$ of $\XX$. The first, and classical one, builds the Picard group as the set of all isomorphism classes of line bundles with the abelian group structure given by the tensor product of $\coh{\XX}$ discussed in the appendix (Section~\ref{sect:Appendix}). The other three incarnations are byproducts of divisor theory and just rely on the structure of $\coh{\XX}$ as a noetherian, hereditary abelian category with Serre duality.

For the second incarnation of the Picard group let $\Pic{\XX}$ be the pointed set of isomorphism classes of line bundles with base element $\OXX$.

Let $U$ be a finite length sheaf that is a homomorphic image of the structure sheaf $\OXX$. Then $U=\bigoplus_{x\in X} U_x$, where $U_x$ is either zero or else an indecomposable sheaf concentrated in $x$ with top $S_x$. We associate to $U$ the divisor $|U|=\sum_{x\in \XX}\ell(U_x)x$ in $\Div{\XX}$, where $\ell$ denotes the length. The exact inclusion $\coh{X}\hookrightarrow\coh{\XX}$ induces an equivalence of the Serre quotients $\coh{X}/\cohnull{X}$ and $\coh{\XX}/\cohnull{\XX}$. Hence the \emph{function field} $K=\funct{\XX}$ of $\XX$, defined the endomorphism ring of $\OXX$ in the Serre quotient $\coh{\XX}/\cohnull{\XX}$, is isomorphic to the function field $\funct{X}$ of $X$. For a non-zero member $f=s^{-1}u$ of $\funct{\XX}$ we define its divisor $\divi_\XX(f)$ as $|\coker{u}|-|\coker{s}|$. By definition, the \emph{divisor class group} $\Cl{\XX}$ of $\XX$ is the cokernel of the divisor map $\divi_\XX:\funct{\XX}^*\to \Div{\XX}$.

\begin{proposition}\label{prop:ActionDivisorGroup}
The divisor group $\Div{\XX}$ of $\XX$ acts on the category $\coh{\XX}$ of coherent sheaves on $\XX$, preserving its full subcategories $\cohnull{\XX}$ of finite length, $\vect{\XX}$ of vector bundles and $\Ll(\XX)$ of line bundles on $\XX$.

The action of the group $\Div{\XX}$ on the set $\Pic{\XX}$ of isomorphism classes of line bundles is transitive and the assignment $D\mapsto \OXX(D)$ induces a bijection $\Cl{\XX}\to \Pic{\XX}$.
\end{proposition}
\begin{proof}
For the first statement, we refer to \cite[Section 10.3]{Lenzing:2007} and \cite[Section 3]{Lenzing:Pena:1999}.
We are next going to show that any two line bundles $L$, $\bar{L}$ belong to the same $\Div{\XX}$-orbit. Let $x_0$ be an ordinary point, then $\euform{L}{\bar{L}(nx_0)}>0$ for large $n$. Accordingly, we obtain an exact sequence $0\to L \to \bar{L}(nx_0)\to U\to 0$ with $U$ having finite length $\ell$. If $\ell=0$, we are done; otherwise we choose a simple subobject $S$ of $U$ which yields a line bundle $L'$ with $L\subset L'\subseteq \bar{L}$ together with an exact sequence $0\to L \to L' \to S \to 0$. If $S$ is concentrated in $x$, then $L'=L(x)$ and $\bar{L}/L(x)$ has length $\ell-1$. By induction $L(x)$ and $\bar{L}$ belongs to the same $\Div{\XX}$-orbit showing that the action of $\Div{\XX}$ on $\Pic{\XX}$ is transitive.

It remains to show that the stabilizer group of divisors $D$ with $\OXX(D)=\OXX$ agrees with the image of the divisor map $\divi_\XX$: Let $D$ a divisor with $\OXX(D)=\OXX$. We write $D$ as $D_{+}-D_{-}$, where $D_{+}=\sum_{x\in X}n_x x$ (resp.\ $D_{-}=\sum_{x\in X}m_x x$) with all $n_x,m_x\geq0$. The assumption $\OXX(D)=\OXX$ translates into $\OXX(-D_{+})=\OXX(-D_{-})$ yielding short exact sequences $0\to \OXX(-D_{+})\up{u} \OXX \to U \to 0$ and $0\to \OXX(-D_{-})\up{s}\OXX\to V \to 0$. This yields a member $f=s^{-1}u$ of $\funct{\XX}^*$ with divisor $\divi_\XX(f)=|U|-|V|=D_{+}-D_{-}=D$.
\end{proof}
Using the bijection $\Cl{\XX}\to\Pic{\XX}$, $[D]\mapsto \OXX(D)$, to equip the pointed set $\Pic{\XX}$ with the structure of an abelian group, we obtain the second incarnation for the Picard group. Moreover, this leads to the weighted divisor sequence of $\XX$:
\begin{corollary}[Weighted divisor sequence]
The sequence
\begin{equation}\label{eqn:WeightedDivisorSequence}
1 \to k^*\to k(\XX)^* \up{\divi_\XX} \Div{\XX} \up{\varphi} \Pic{\XX} \to 0
\end{equation}
is exact and identifies $\Pic{\XX}$ with the divisor class group of $\XX$.~\qed
\end{corollary}

We next introduce a third incarnation of the Picard group by defining
the \emph{Picard group} $\Pic{\XX}$ of $\XX$ as the subgroup of $\Aut{\coh\XX}$, the group of isomorphism classes of self-equivalences of $\coh{\XX}$, which is generated by all (isomorphism classes of) point-shifts $\si_x$,  $x\in X$. We note that $\Aut{\XX}$, defined as the subgroup of automorphisms of $X$ commuting with the weight function, can be identified with the subgroup of (isomorphism classes) of self-equivalences of $\coh{\XX}$, that are fixing the structure sheaf $\OXX$. This follows as in \cite{Lenzing:Meltzer:2000}.

Since the shift actions $\si_x$, $\si_y$ commute for each pair $x,y$ of points from $X$, the Picard group is commutative. Its role is clarified by the following statement.
\begin{proposition} \label{prop:AutomorphismGroup}
The group $\Aut{\coh\XX}$ is generated by its subgroups $\Aut{\XX}$ and $\Pic{\XX}$. More precisely, $\Aut{\coh\XX}=\Aut{\XX}\ltimes\Pic{\XX}$.
\end{proposition}
\begin{proof}
Let $\Ll(\XX)$ be denote the full subcategory of $\coh{\XX}$ formed by all line bundles. Since automorphisms of $\coh{\XX}$ preserve the rank, the action of $G=\Aut{\coh{\XX}}$ restricts to an action of $G$ on $\Ll(X)$. This action is transitive (up to isomorphism). Moreover, the stabilizer group of $\OXX$ is just $\Aut{\XX}$. This proves that $\Aut{\Ll(X)}$ is generated by $\Pic{\XX}$ and $\Aut{\XX}$. Now, each self-equivalence of $\Ll(\XX)$ extends uniquely to a self-equivalence of $\coh{\XX}$, implying the claim on the structure of $\Aut{\coh{\XX}}$.
\end{proof}
The Picard group is the dominant part of $\Aut{\coh{\XX}}$, since in most cases $\Aut{\XX}$ is finite. Indeed only for an unweighted elliptic curve or for a weighted projective line $\XX$ with at most two weighted points, the group $\Aut{\XX}$ is infinite:
weighted projective lines with at least three weighted points have only finitely many automorphisms~\cite{Lenzing:Meltzer:2000}. Also, for a point of a smooth elliptic curve $X$, there are only finitely many automorphisms of $X$ that are fixing $x$, implying that elliptic curves that are weighted non-trivially have a finite automorphism group. Finally we use that by Harnack's theorem the automorphism group of a smooth projective curve $X$ of negative Euler characteristic is always finite.

Next, we discuss the relationship between the divisor groups $\Div{X}$ and $\Div{\XX}$ respectively the relationship between their class groups $\Cl{X}$ and $\Cl{\XX}$, that is of the Picard groups $\Pic{X}$ and $\Pic{\XX}$.

Let $H$ be an abelian group together with a sequence $(h_1,h_2,\ldots,h_t)$ of members of $H$ and a corresponding sequence $(p_1,p_2,\ldots,p_t)$ of positive integers. Let further $(e_1,\ldots,e_t)$ be the standard basis of $\ZZ^t$ and $U$ be the subgroup of $H\oplus \ZZ^t$ generated by the elements $p_ie_i-h_i$. We denote the factor group $\bar{H}=(H\oplus\ZZ^t)/U$ by $\bar{H}=H[\frac{h_1}{p_1},\frac{h_2}{p_2},\ldots,\frac{h_t}{p_t}]$, where $\frac{h_i}{p_i}$ denotes the coset $Ue_i$ from $\bar{H}$, and call $\bar{H}$ the group obtained from $H$ by \emph{adjoining the universal fractions $\frac{h_i}{p_i}$}, $i=1,2,\ldots,t$.

\begin{theorem}[Universal fractions] \label{thm:UniversalFractions}
Let $\XX=(X,w)$ be a weighted projective curve and $x_1,x_2,\ldots,x_t$ its weighted points with $p_i=w(x_i)$. Then $\Div{\XX}$ (respectively $\Cl{\XX}$) arises from $\Div{X}$ (resp.\ from $\Cl{\XX}$) by adjoining the universal fractions $\frac{x_i}{p_i}$, $i=1,\ldots,t$. Accordingly, also $\Pic{\XX}$ arises from $\Pic{X}$ by adjoining universal fractions.
\end{theorem}

\begin{proof}
We interpret $\coh{X}$ as the full exact subcategory of $\coh{\XX}$ obtained as the right perpendicular category to the system $\Ee$ of exceptional sheaves $(\tau^j S_{x_i})$, $j=1,\ldots,p_i-1$, $i=1,\ldots,t$, compare \cite{Geigle:Lenzing:1991}. We observe that each simple sheaf $S_x$ of $\coh{X}$ under this full, exact embedding $j:\coh{X}\to \coh{\XX}$ is sent to the unique indecomposable sheaf $U_x$ of $\coh{\XX}$ that is concentrated in $x$, has length $w(x)$ and satisfies $\Hom{\OXX}{U_x}=k$. Further, the structure sheaf $\OX$ of $\coh{X}$, when viewed as a member of $\coh{\XX}$ takes the role of the structure sheaf there (notation: $\OXX$). Passing to divisor groups $j$ induces an inclusion $\Div{X}\to \Div{\XX}$, $\sum_{x\in X}n_x x\mapsto \sum_{x\in X}w(x)n_x x$ for the divisor groups. As is easily checked, the following diagram with exact rows is commutative and has also exact columns
\begin{equation*}
  \xymatrix@R12pt@C18pt{
       &                          &0\ar[d]                 &0\ar[d]&\\
0\ar[r]&  {K^*/k^*}  \ar[r]^{\divi_X}\ar@{=}[d] &\Div{X}\ar[d] \ar[r] &  \Cl{X}\ar[d]\ar[r]&0\\
0\ar[r]&  K^*/k* \ar[r]^{\divi_\XX} &\Div{\XX}\ar[d]\ar[r] &  \Cl{\XX}\ar[d]\ar[r]&0\\
&&\prod_{i=1}^{t}\ZZ_{p_i}\ar[d]\ar@{=}[r]    &\prod_{i=1}^{t}\ZZ_{p_i}\ar[d]&\\
&&   0                                  & 0 &
  }
\end{equation*}
 which immediately implies the claim.
\end{proof}
\begin{example}
If $\XX$ is a weighted projective line $\PP^1\wt{p_1,p_2,\ldots,p_t}$ then the Picard group of $\XX$, written additively, is obtained by adjoining universal fractions as $\Pic{\XX}=\ZZ[\frac{1}{p_1},\frac{1}{p_2},\ldots,\frac{1}{p_t}]$. This group is obviously isomorphic to the rank-one abelian group $\LL(p_1,p_2,\ldots,p_t)$ on generators $\vx_1,\vx_2,\ldots,\vx_t$ with relations $p_1\vx_1=p_2\vx_2=\ldots,p_t\vx_t$ showing up in the theory of weighted projective lines~\cite{Geigle:Lenzing:1987}. These groups usually have torsion; thus a universal fraction $\frac{1}{p}$ should not be confused with the rational number given by the same expression.
\end{example}

\subsection{The localization sequence}
Following the approach in Section~\ref{sect:DivisorSequence}  the exact localization K-theory sequence, due to Quillen~\cite{Quillen:1973}, transforms the `exact sequence' $0\to \cohnull{\XX}\to \coh{\XX}\to \mmod{\funct{X}}\to 0$ of abelian categories to an exact sequence
\begin{equation}\label{eqn:LocalizationSequence}
  1\to \funct{\XX}^*/k^* \up{\divi}\Knull{\cohnull\XX}\up{\nat}\Knull{\coh{\XX}}\up{\rk{}}\ZZ\to 0,
\end{equation}
where $\divi$ sends a fraction $f=s^{-1}u$ to $[\coker{u}]_0-[\coker{s}]_0$, and $\nat$ sends the class $[U]_0$ from $\Knull{\cohnull{\XX}}$ to the class $[U]$ in $\Knull{\coh{\XX}}$. We are now going to relate this sequence to the divisor theory for $\XX$, yielding a fourth incarnation of the Picard group $\Pic{\XX}$. We recall that $\Knullprime{\XX}$ is the kernel of the rank homomorphism $\Knull{\XX}\to \ZZ$, and also the subgroup of $\Knull{\XX}$ generated by the classes of simple sheaves. Further let $V$ (resp.\ $V_0$) be the subgroup of $\Knullprime{\XX}$ (resp.\ $\Knull{\cohnull{\XX}}$) generated by all elements $[\tau^S_x]-[S_x]$ (resp.\  $[\tau^jS_x]_0-[S_x]_0$), $j=1,\ldots,w(x)$, where $x$ runs through all weighted points of $\XX$ and $S_x$ is a simple sheaf concentrated in $x$.

\begin{proposition}
The following diagram is commutative with exact rows and columns
\begin{equation}\label{eqn:DivisorVersusLocalization}
  \xymatrix@R12pt@C18pt{
         &                                 &  0\ar[d]          &0\ar[d] \\
       &                                   &V_0\ar[d]\ar@{=}[r]&V \ar[d]            \\
1\ar[r]&\funct{\XX}^*/k^*\ar@{=}[d]\ar[r]^{\divi}&\Knull{\cohnull{\XX}}\ar[d]\ar[r]^{\nat}&\Knullprime{\XX}\ar[d]\ar[r]&0\\
1\ar[r]&\funct{\XX}^*/k^*\ar[r]^{\divi_\XX}&\Div{\XX}\ar[r]\ar[d]&\Cl{\XX}\ar[d]\ar[r]&0\\
       &                                   &  0          &0 \\
        }
\end{equation}
\end{proposition}
\begin{proof}
The commutativity of the diagram is easily checked. Moreover, by Proposition~\ref{prop:RClass}, the induced map from $V_0$ to $V$ is an isomorphism.
\end{proof}

\begin{corollary}
We have a direct decomposition $\Knull{\XX}=\Knullprime{\XX}\oplus \ZZ[\OXX]$. Moreover, there is an exact sequence $0\to U \to \Knullprime{\XX}\to\Cl{\XX}\to 0$, representing $\Cl{\XX}$, hence $\Pic{\XX}$, as a factor group $\Knull{\coh{\XX}}/(U\oplus \ZZ[\OXX])$ of the Grothendieck group $\Knull{\XX}$.~\qed
\end{corollary}
We thus obtain a fourth incarnation of the Picard group as factor group of the Grothendieck group of $\coh{\XX}$. As the weighted projective lines show, it is---unlike in the non-weighted case--- usually not possible to represent the Picard group as a subgroup of the Grothendieck group.

\appendix
\section{Multiplicative structure}\label{sect:Appendix}
Let $\Hh$ be the category of coherent sheaves on a weighted projective curve $\XX=X\begin{pmatrix}p_1,p_2,\ldots,p_t\\x_1,x_2,\ldots,x_t\end{pmatrix}$ obtained by iterated $p$-cycle construction from the category $\coh{X}$ of coherent sheaves on the smooth projective curve $X$. We sketch in this appendix how to introduce a $\ZZ$-graded sheaf theory such that the attached category of coherent sheaves is equivalent to $\Hh$ and, moreover, equips $\Hh$ with a commutative tensor product satisfying $\OXX(x)\otimes\OXX(y)=\OXX(x+y)$, functorially. In particular, this turns the isomorphism classes of line bundles on $\XX$ into a commutative group, the \emph{Picard group} $\Pic{\XX}$ of $\XX$. Moreover, the graded sheaf theory allows to form the \emph{determinant homomorphism} $\det:\Knull{\coh{\XX}}\to \Pic{\XX}$, sending the class $[E]$ of a vector bundle $E$ of rank $r$ to its $p$-th exterior power $\bigwedge^pE$.

The procedure follows \cite{Reiten:VandenBergh:2002}, which the reader should consult for details. Let $x_0$ be an ordinary point. We form the divisor $D=mx_0+\sum_{i=1}^{t}x_t$. Functoriality of the shift functor $X\mapsto X(D)$ turns the direct sum
\begin{equation*}
   R=\bigoplus_{n\in\ZZ}\Hom{\OXX}{\OXX(nD)}
 \end{equation*}
into a $\ZZ$-graded algebra with finite dimensional components $R_n=\Hom{\OXX}{\OXX(nD)}$. The components $R_n$ can be naturally seen as subspaces of the function field $\funct{\XX}=\funct{X}$, identifying $R$ by means of the map $\sum_{n\in\ZZ}{r_n}\mapsto \sum_{n\in\ZZ}{r_nY^n}$ with a subalgebra of the polynomial algebra $\funct{X}[Y]$ in one variable over $\funct{X}$. Taking noetherianness of $\Hh$ into account, this shows that $R$ is a positively $\ZZ$-graded affine $k$-algebra, provided we choose $m$ sufficiently large.
\begin{proposition}
The following properties hold:
\begin{enumerate}
  \item For each $A$ in $\Hh$, the direct sum $F(A)=\bigoplus_{n\geq0}\Hom{\OXX(-nD)}{A}$ is a finitely generated positively graded $R$-module.
  \item The attachment $A\mapsto F(A)$ induces an equivalence between $\Hh$ and the Serre quotient $\Cc$ of  the category $\modgr{\ZZ_{\geq0}}{R}$ of finitely generated positively $\ZZ$-graded modules modulo its Serre subcategory of graded modules of finite length.
  \item The graded tensor product $\otimes$ (resp.\ graded $r$-th exterior power $\bigwedge^r$) induce corresponding operations on $\Cc$ and therefore $\Hh$, having the usual properties.
  \item The set of $\Pic{\XX}$ of isomorphism classes of line bundles is a group with respect to the tensor product operation.
\end{enumerate}
\end{proposition}
\begin{proof}
Concerning the first two statements, we refer to \cite{Reiten:VandenBergh:2002}. We just remark that the key point is that, by the choice of the divisor $D$, we obtain for any simple sheaf $S$ an \emph{infinite} set of positive integers $n$ such that $\Hom{\OXX(-nD)}{S}$ is non-zero.

Concerning the third item, the proofs from \cite{Geigle:Lenzing:1986} can be copied.
\end{proof}
\begin{corollary}
Each line bundle $L$ on $\XX$ is determined, up to isomorphism, by its class $[L]$ in the Grothendieck group.
\end{corollary}
\begin{proof}
Due to the multiplicative structure on $\coh{\XX}$, the classical proof of Proposition~\ref{prop:LineBundleClass} turns over.
\end{proof}
Unlike exceptional objects, line bundles are not determined by their class in the reduced Grothendieck group $\rKnull{\XX}$.

\bibliographystyle{amsplain}

\providecommand{\bysame}{\leavevmode\hbox to3em{\hrulefill}\thinspace}
\providecommand{\MR}{\relax\ifhmode\unskip\space\fi MR }
\providecommand{\MRhref}[2]{%
  \href{http://www.ams.org/mathscinet-getitem?mr=#1}{#2}
}
\providecommand{\href}[2]{#2}

\end{document}